\documentclass{amsart}
\usepackage{amssymb,amsmath,amsthm}
\usepackage{mathrsfs}
\usepackage{hyperref}
\usepackage{mathtools}
\usepackage{breqn}
\usepackage{color}
\usepackage{verbatim}
\usepackage{comment}

\renewcommand{\[}{\begin{equation}\begin{aligned}}
\renewcommand{\]}{\end{aligned} \end{equation}}

\newcommand{\abs}[1]{\lvert#1\rvert}
\newcommand{\nrm}[1]{\lvert\lvert#1\rvert\rvert}
\newcommand{\R}{\mathbb{R}}
\newcommand{\D}{\mathbb{D}}
\newcommand{\iprod}{\mathbin{\raisebox{\depth}{\scalebox{1}[-1]{$\lnot$}}}}
\newcommand{\h}{\mathfrak{h}}
\newcommand{\N}{\breve{N}}
\newcommand{\F}{\mathcal{F}}
\newcommand{\gr}{\text{grad}}
\newcommand{\dbd}{\bar{\partial}\partial}

\newtheorem{thm}{Theorem}
\newtheorem{prop}[thm]{Proposition}
\newtheorem{lemma}[thm]{Lemma}

\newtheorem{cor}[thm]{Corollary}

\theoremstyle{remark}

\theoremstyle{definition}
\newtheorem{definition}[thm]{Definition}

\numberwithin{equation}{section}
\numberwithin{thm}{section}

\author{Ethan Lane Addison}
\address{Mathematics Department, Stony Brook University, Stony Brook, NY 11794}
\email{addison@math.stonybrook.edu}

\title{A Product Model for Generalizing Poincaré-Type Kähler Metrics}
\date{}

\begin{document}

\begin{abstract}
We begin by defining a type of Kähler metric near the zero section of a trivial holomorphic open disk bundle $N$ over a compact Kähler manifold $X$ by incorporating flows generated by holomorphic vector fields on $X$. These metrics are then shown to deviate exponentially from Poincaré-type metrics on $N\setminus X$ in terms of the log-polar distance from $X$ in $N$. Lastly we see that they arise naturally when perturbing classes containing Poincaré-type Kähler metrics of constant scalar curvature to obtain nearby cscK metrics even when the perturbed class on $X$ does not admit a cscK metric.
\end{abstract}
\maketitle

\section{Introduction}
A persistent theme in complex geometry is the search for preferred or canonical choices of metrics on complex manifolds. In the case that the manifold is Kähler, one can asked more pointed questions such as whether or not a canonical metric can be found in a given Kähler class. By seeking minimizers of the $L^2$-norm of the curvature tensor in such a class, E. Calabi introduced in \cite{Cal} the notion of an extremal metric, a generalization of Kähler metrics of constant scalar curvature, respectively abbreviated to extK and cscK metrics. Examples of the longevity of the search for canonical metrics stretch back to the classical Uniformization Theorem, up through Yau's Theorem \cite{Yau} which resolved a conjecture of Calabi, continuing to the modern day with the replete understanding of the Kähler-Einstein setting elucidated by Chen-Donaldson-Sun \cite{CDS}.

A fundamental property of compact Kähler manifolds shown by LeBrun-Simanca in \cite{LBS} is the openness in the Kähler cone of the set of Kähler classes which contain extremal metrics. In fact, if the manifold has discrete automorphism group, then this openness is upheld even in the case of cscK metrics. When one relaxes compact Kähler to complete Kähler, properties like this LeBrun-Simanca openness are less clear, an issue compounded by restricting attention to only specific types of metrics, such as those defined by their asymptotic behavior near the ends. In cases wherein the ends are composed of the components of a complex hypersurface of some compactified manifold --- as arise in studying quasiprojective varieties or are conjectured to emerge \cite{Dn} in seeking a minimizing sequence for the Calabi functional when a minimizer does not exist --- one can define complete metrics with cusp-like singularities along the hypersurface. 

Let $(X,\sigma)$ be a compact cscK manifold with complex dimension $n$.  After removing the zero section, the open unit disk bundle $N= \D \times X$ admits such a complete cscK metric $\omega$ by taking the product of $\sigma$ and the standard Poincaré metric on the punctured disk. Following the definition given by Auvray \cite{Auv}, other metrics which are quasi-isometric to $\omega$ and have bounded derivatives at all orders with respect to $\omega$ are considered metrics of Poincaré-type, abbreviated here as PT. Despite initially seeming like a loose categorization, Auvray also illustrates by considering only the local behavior near $X$ (which we hereafter identify with the zero section of $N$), regardless of the nature of the normal bundle or tubular neighborhood, that the asymptotic properties of such metrics are tightly controlled when said metric is extremal; we note here that recent investigations of these metrics tend toward the extK and cscK setting since the special case of Kähler-Einstein PT metrics has been well documented by Yau \cite{Yauu}, Kobayashi \cite{Koby}, and Wu \cite{Wu}. 

The asymptotics imply that if $\tilde{\sigma}$ is a smooth real $(1,1)$-form on all of $N$ such that $\tilde{\sigma}+d\beta$ is PT, then this PT metric is extK (respectively cscK) only if it induces an extK (respectively cscK) metric on $X$ in the class $[\tilde{\sigma}|_X]$. Though PT metrics enjoy many of the same properties of Kähler metrics on compact manifolds, this above restriction belies a sensitivity to the geometry of the hypersurface $X$ that prevents perturbations of cscK PT metrics from being cscK. From LeBrun-Simanca openness, if the class $[\sigma]$ on $X$ is perturbed to another nearby class, then the resulting class will at least admit an extremal metric. However, if this resulting metric is not cscK, which could be the case if $X$ admits nontrivial holomorphic vector fields, then the perturbation of $[\sigma]$ cannot be extended to a cscK PT metric on the punctured disk bundle in a way that yields another cscK PT metric, as Sektnan in \cite{Sek} shows the extremal vector field of an extK PT metric must be holomorphically extendable across $X$. One can infer then that PT metrics, while rich in examples and desirable properties, may be insufficient on their own to establish a consummate theory of complete extK metrics, as suggested by the non-PT complete extremal metrics constructed by Apostolov-Auvray-Sektnan \cite{AAS} in the toric setting. Incidentally, as the hypersurfaces in this case are toric themselves, they have dense open subsets which are trivially embedded into the ambient space, looking almost everywhere like the setting we explore presently.

By incorporating the vector fields directly into a Kähler distortion potential on $\breve{N} = N\setminus X$, one can compensate for extra elements in the kernel of the Lichnerowicz operator for the PT metric and define a new complete metric which is not PT but still well-behaved and which can moreover be chosen to be cscK when one is starting from cscK initial data. The construction of such metrics is the basis of this article; given an appropriate vector field $V$ on $X$, the associated potential which we refer to as the \emph{gnarl} associated to $V$ by $\sigma$ will be notated $\Psi_V$. The result of adding a Kähler metric to the complex Hessian of $\Psi_V$ will be referred to as a \emph{gnarled} metric. Using the real harmonic $(1,1)$-forms on $X$ as the parameter of perturbation, our main result can be stated in these terms:

\begin{thm}\label{thm:T1}
  For $(X,\sigma)$ a compact cscK manifold, there is a neighborhood $U$ about $0$ in $\mathcal{H}^{1,1}_{\sigma}(X,\R)$ such that for every $\eta \in U$ there exists a real holomorphic vector field $V_\eta$ on $X$ and a cscK gnarled Poincaré-type metric on $\breve{N}$ associated to $V_{\eta}$ by a metric in $[\sigma + \eta]$ which pulls back to this cohomology class under constant sections of $\breve{N}$.
\end{thm}

This result is in a similar vein to the compact case explored by LeBrun-Simanca, and while one can consider it a local result, the gnarl can be included into any manifold for which $X$ is a divisor with holomorphically trivial normal bundle with the use of a suitable cutoff. The following two sections serve to explicate the construction of the gnarled metric in this simple product case and then to provide a proof of the main result by the introduction of a modified scalar curvature operator which we show can be inverted to solve the necessary equation guiding constant curvature for our gnarled metric.

\section{Construction of the Metric}\label{sec:construction}

In order to construct the gnarl, we take $z$ as the standard coordinate on $\D$ and define a ``logarithmic-polar"  coordinate system on $\D^* = \D \setminus \{0\}$ by using the standard angular coordinate $\theta$ along with a radial function $\tau = \ln(-\ln\abs{z})$ which goes to $\infty$ near $X$. The cusp metric we use for $\D^*$ has $\tau$ as a Kähler potential and will be written $dd^c\tau$, where we define the operator $d^c = JdJ^{-1}$. Here the complex structure $J$ is a real endomorphism acting on covariant tensors by precomposition, thus on $k$-forms, $J^{-1} = (-1)^kJ$. Under the decomposition $d = \partial + \bar{\partial}$, we have that $d^c = i(\partial - \bar{\partial})$, and therefore the complex Hessian operator can be written as $dd^c = 2i\dbd$. From the definition of the new coordinates, we can see $d^c\tau = e^{-\tau}d\theta$ as well as the ``cusp" nature of the volume form $e^{-\tau}d\theta\wedge d\tau$ upon approach to the zero section $X$.

Since any obstruction to LeBrun-Simanca openness for PT metrics in this context will arise from holomorphic vector fields on $X$, we now note that by Matsushima \cite{Mats} and Lichnerowicz \cite{Lich} the following fact that $X$ admitting a cscK metric implies that the space of holomorphic vector fields on $X$ is a complex Lie algebra which splits into
$$H^0(X,T_X) = \mathfrak{a}\oplus\mathfrak{h}\oplus J\mathfrak{h}.$$
Here $\mathfrak{a}$ represent fields parallel with respect to the Levi-Civita connection while $J\mathfrak{h}$ is a subspace of real vector fields which are Hamiltonian with respect to the symplectic form $\sigma$. From the relationship between the symplectic form and the metric tensor on $X$, each $V \in \mathfrak{h}$ is therefore a gradient vector field, i.e. $V = \gr f_V$ for a real function $f_V$, which has an associated time-$t$ flow notated by $F_{tV}$. These are the only vector fields we need to consider since the kernel of the Lichnerowicz operator contains only holomorphy potentials like $f_V$, and the nontrivial parallel fields cannot be gradients because they do not have zeroes on $X$.

Starting from such a $V$, we can define a function $\psi_{tV} \in \mathcal{C}^{\infty}(\R\times X)$ as the solution to the differential equation:
\begin{equation}\label{gnarl1}
\frac{d}{dt}\psi_{tV} = F^*_{tV}f_V, \text{ and } \psi_{0} = 0.
\end{equation}
Since the flow is a path in Aut$_0(X)$, $F^*_{tV}\sigma$ remains in the same de Rham and Dolbeault class of $\sigma$, and so by the $dd^c$-lemma, we see that the $\psi_{tV}$ satisfies 
\begin{equation}\label{ddcl}
dd^c\psi_{tV} = F^*_{tV}\sigma - \sigma.
\end{equation}

This function serves as the starting point of the gnarl. Recalling that we consider $N$ as a bundle over $X$, we set $\nu:N \twoheadrightarrow X$ as the projection map. Then $\tau \times \nu$ is a map from $\N$ to $\R\times X$, so $\psi_{tV}$ can be pulled back to the punctured disk bundle.

\begin{definition}\label{defn:gnarl}
   Given $V \in \mathfrak{h}$, the gnarl associated to $V$ is the function given by $\Psi_V = (\tau \times \nu)^*\psi_{tV}$. For a given Kähler metric $\varpi$ on $\N$, we set $\varpi_V = \varpi + dd^c\Psi_V$.
\end{definition}

Recalling that the product metric on $\N$ is given by $\omega = dd^c\tau + \nu^*\sigma$, where we now suppress the pullback by $\nu$ unless necessary to include it, in line with Definition 2.1 we set $\omega_V = \omega + dd^c\Psi_V$. This is a well-defined real $(1,1)$-form on $\N$, but the first task will be to show that it can actually be considered a metric, i.e. that it is a positive form. This may in fact not be true for very large choices of vector field $V$, but the following argument will make clear that so long as $V$ is small enough with respect to the starting data, then $\omega_V$ is a genuine Kähler metric. As we are interested eventually in a perturbation problem, only a small neighborhood of $0$ in $\mathfrak{h}$ will be of great concern, so this will be sufficient for the proof of Theorem 1.1. What is more, that $\sigma$ is cscK is not relevant to this positivity.

\begin{prop}\label{prop:pos}
Let $(\tau\times\upsilon)^*\psi_{tV}^\varphi$ be the gnarl associated to $V$ by a metric $\sigma+dd^c\varphi$ on $X$ for some $\varphi$. Then there exists a neighborhood $\mathcal{W} \subset \mathfrak{h}$ about $0$ such that for any $V \in \mathcal{W}$, $\omega + dd^c(
\upsilon^*\varphi + (\tau\times\upsilon)^*\psi_{tV}^\varphi)$ is a symplectic form on $\N$.
\end{prop}

\begin{proof}
Without loss of generality, we show this explicitly only for $\omega_V$. Working on the punctured disk bundle is less convenient than that of the corresponding upper-half-plane bundle, so we introduce the covering map
 \begin{equation}\label{upsi}
     \upsilon:\mathbb{H} \times  X \twoheadrightarrow \N \text{ via } (\zeta,x) \mapsto (e^{i\zeta},x).
 \end{equation}
 
 This map pulls $\tau$ up to a new variable $w = \ln \mathfrak{Im}\zeta$, and then the entire gnarled form $\omega_V$ for some unspecified $V$ can be written out as
 \begin{equation*}
     \upsilon^*\omega_V = \sigma + dd^c\left(w+\psi_{wV}\right).
 \end{equation*}

What is more, since both bundles are trivial, the flow $F_{tV}$ can be extended to $\N$ trivially and can be lifted up to $\mathbb{H} \times X$ as a $t$-dependent family of diffeomorphisms $\tilde{F}_{tV}$ according to
$$\tilde{F}_{tV}(\zeta,x) = (e^{-t}\zeta, F_{tV}(x)).$$

Note that for any given $t$, $\tilde{F}_{tV}$ is actually a biholomorphism. With the goal of showing these form a family of isometries for $\omega_V$, we must illustrate a handy property of the gnarl itself. Restricting our attention back to $X$ again for a moment, observe that for a choice of specific time $T$,
$$F^*_{TV}dd^c\psi_{tV} = F^*_{(T+t)V}\sigma - F^*_{TV}\sigma = dd^c\left(\psi_{(T+t)V}-\psi_{TV}\right).$$
Therefore there is a function $\kappa_T(t)$ that allows us to equate
$$F^*_{TV}\psi_{tV} = \psi_{(T+t)V}-\psi_{TV} + \kappa_T(t).$$
We would like to show this additional term is in fact 0. Initially we see that for $t = 0$, $\kappa_T(0) = 0$. Additionally, taking the first derivative yields
$$\kappa'_T(t) = \frac{d}{dt}\psi_{(T+t)V}-F^*_{TV}\frac{d}{dt}\psi_{tV} = F^*_{(T+t)V}f_V - F^*_{TV}F^*_{tV}f_V = 0.$$
Hence $\kappa_T(t) = 0$ for all $T$ and $t$, and we have a clear understanding of how the flow along $V$ affects the associated gnarl, i.e. according to
\begin{equation}\label{removeH}
    F^*_{TV}\psi_{tV} = \psi_{(T+t)V}-\psi_{TV}.
\end{equation}

Returning to the cover, this formula arises when flowing the metric $\upsilon^*\omega_V$:
\begin{equation*}
    \tilde{F}^*_{TV}(\upsilon^*\omega_V) = F^*_{TV}\sigma + dd^cw + dd^c\left(\psi_{(w+T-T)V}\right)-dd^c\left(\psi_{TV}\right).
\end{equation*}
Note that since $\psi_{TV}$ is independent of the fiber variable, on the cover we explicitly have $$\upsilon^*dd^c\psi_{TV} = \tilde{F}^*_{TV}(\upsilon^*\sigma) - \upsilon^*\sigma$$ Therefore we arrive at some cancellations from equation \ref{ddcl}:
$$\tilde{F}^*_{TV}\upsilon^*(\omega + dd^c\psi_{\tau V}) = \sigma + dd^c(w+\psi_{wV}) = \upsilon^*(\omega + dd^c\psi_{\tau V}).$$

On any given strip, say $\mathcal{Q} = \{1 \leq \mathfrak{Im}\zeta \leq 2\} \times X$, we can certainly choose $V$ to be small enough so that the above form is positive since the gnarl is independent of the angular coordinate $\theta$ and thus the whole form only varies on a compact set. However, for any $p \in \mathbb{H}\times X$, there exists a $T$ such that $\tilde{F}_{TV}(p) \in \mathcal{Q}.$ Given that we now have chosen $V$ such that $\upsilon^*(\omega + dd^c\psi_{\tau V})$ is positive in $\mathcal{Q}$ and the form has been shown to be preserved by the map $\tilde{F}_{TV}$, it follows $\upsilon^*(\omega + dd^c\psi_{\tau V}) > 0$ on $\mathbb{H} \times X$. Since it is a local diffeomorphism, $\upsilon$ cannot pull a nonpositive form back to a positive form, and we see that $\omega_V$ defines a Kähler metric on $\N$. As this procedure can be done for any choice of direction $V/\nrm{V}_{\mathcal{C}^0(X,\sigma)}$, we arrive at an open neighborhood $\mathcal{W}$ of 0, as desired.
\end{proof}

A remark we make here that will be vital to the proof that follows arises from noticing that we can study the scalar curvature of $\upsilon^*\omega_V$  by looking at a single slice, say $\{2i\}\times X$. Here we do not mean the scalar curvature of the restricted metric on the slice, but rather the restriction of the scalar curvature as a function; the biholomorphisms $\tilde{F}_{tV}$ can move this function to every other slice above this one as $t$ varies, and in the case that this function happens to be a constant, the value of $t$ is irrelevant. Certainly this remains true on $\N$ since $\upsilon$ is a local diffeomorphism. Recall that Proposition \ref{prop:pos} holds for any starting $\sigma + dd^c\varphi$ on $X$, and therefore the procedure to prove Theorem \ref{thm:T1} will be to search for a function $\varphi$ on $X$ by looking at the restriction of the scalar curvature of $\omega + dd^c\nu^*\varphi$ to $\{e^{-2}\}\times X \hookrightarrow \N$, essentially converting the problem to one in which we must only concern ourselves with functions on $X$. 

More than simply showing $\omega_V$ is a metric, the covering space argument can be used to get $\mathcal{C}^k$ bounds on $\omega_V$ which are derived from estimating the size of the gnarl with respect to the PT metric $\omega$. Going forward, we use the asymptotic notation $\lesssim$ to express eventual subjugation of the left-hand side by a constant multiple of the right-hand side with respect to either $t,\tau,$ or $T$ heading towards infinity. Which of these variables is suggested should be clear from context and its presence in an exponential expression.

\begin{prop}\label{prop:P23}
For every $k \geq 0$, there exists an $\varepsilon_k > 0$ such that
$$\nrm{\Psi_V}_{\mathcal{C}^k(\omega)} \lesssim e^{\varepsilon_k \tau}.$$
\end{prop}

\begin{cor}
For every $k \geq 0$, there exists an $\tilde{\varepsilon}_k > 0$ such that
$$\nrm{\omega_V}_{\mathcal{C}^k(\omega)} \lesssim e^{\tilde{\varepsilon}_k \tau}.$$
\end{cor}
\begin{proof}
By using the above proposition,
$$\nrm{\nabla^k\omega_V}_\omega \lesssim \nrm{\nabla^k\omega}_\omega + \nrm{\nabla^{k+2}\Psi_V}_\omega \lesssim e^{\varepsilon_{k+2}\tau},$$
since the first term is zero. Hence we merely set $\tilde{\varepsilon}_k = \varepsilon_{k+2}$.
\end{proof}
The proposition itself requires a similar but more general and useful result about flows of vector fields on compact manifolds.

\begin{lemma}\label{flowlem}
 Let $\beta$ be a smooth tensor on $(X,\sigma)$ and let $W$ be a smooth vector field with associated time-$t$ flow $F_{tW}$. Then for each $k \geq 0$, there exists $c_k > 0$ depending on $W$ and the valence of $\beta$ such that 
 $$\nrm{F_{tW}^*\beta}_{\mathcal{C}^k(X,\sigma)} \lesssim e^{c_kt}\nrm{\beta}_{\mathcal{C}^k(X,\sigma)}.$$
\end{lemma}
\begin{proof}
We induct on $k$. Starting with $k = 0$, we note that $$\nrm{F^*_{tW}\beta}_{\sigma} = \nrm{\beta}_{F^*_{-tW}\sigma} \leq \nrm{F^*_{-tW}\sigma}_{\sigma}\nrm{\beta}_{\sigma} \leq e^{c_0t}\nrm{\beta}_{\sigma}.$$
Here we find $c_0$ through the following ordinary differential inequality:
\begin{equation}\label{nrmflow}
    \partial_t\nrm{F^*_{tW}\sigma}_{\sigma} \leq \nrm{F^*_{tW}\pounds_{W}\sigma}_{\sigma} = \nrm{F^*_{tW}d(W\iprod\sigma)}_{\sigma} \leq c_0\nrm{F^*_{tW}\sigma}_{\sigma},
\end{equation}
since $d(W\iprod\sigma)$ is just a differential form on $X$, so there exists a constant $c_0$ depending on $W$ such that 
$$\nrm{c_0\sigma}_{L^\infty} \geq \nrm{d(W\iprod\sigma)}_{L^\infty}.$$

For the case $k = 1$, we define a 1-form $S_t$ with values in the endomorphism bundle of whatever bundle contains $\beta$ according to $S_t = \nabla - \nabla_t$ where $\nabla_t = F^*_{tW}(\nabla)$ is the Levi-Civita connection of $F^*_{tW}\sigma.$ This will allow us to pass the flow by the covariant derivative with a cost as seen by
\begin{equation}
    \nabla(F^*_{tW}\beta) = \nabla_t(F^*_{tW}\beta) + S_tF^*_{tW}\beta = F^*_{tW}(\nabla\beta) + S_tF^*_{tW}\beta.
\end{equation}
Taking the norm of the first covariant derivative then yields
\begin{equation*}
    \nrm{\nabla(F^*_{tW}\beta)} \leq \nrm{F^*_{tW}(\nabla\beta)} + \nrm{S_tF^*_{tW}\beta} \leq e^{\tilde{c}_0t}\nrm{\nabla \beta} + e^{c_0t}\nrm{S_t}~\nrm{\beta}.
\end{equation*}
Here we had to adjust the rate to a new $\tilde{c}_0$ (found the same way as in the $k=0$ case) in the first term because the valence increased by 1 with the derivative. Now the issue has been reduced to understanding the derivatives of $S_t$.

Set $g$ and $g_t$ to be the Riemannian metrics associated to $\sigma$ and $F^*_{tW}\sigma$. In general depending on the valence of $\beta$, $S_t$ might be several tensor products of the difference between the connections acting on vector fields. Yet, this simplest difference will control the product, so we assume for a moment that $S_t$ is a section of $\Omega^1(TX\otimes T^*X)$ instead of higher tensor powers. Then we find the following ordinary differential inequality of $g$-norms:
\begin{equation*}
    \partial_t\nrm{S_t} \leq \nrm{\partial_tS_t} \leq \nrm{\dot{\nabla}_t}
\end{equation*}
\begin{equation*}
    =\frac{1}2\nrm{g_t^{j\ell}(\nabla_{t,r}\pounds_{W}g_{t,u\ell}+\nabla_{t,u}\pounds_Wg_{t,r\ell}-\nabla_{t,\ell}\pounds_Wg_{t,ur})F^*_{tW}(\partial_j\otimes dx^u \otimes dx^r)},
\end{equation*}
a consequence of differentiating the Christoffel symbols of $g_t$ with respect to $t$. Since the flow commutes with the Lie derivative, we can simplify this expression to
\begin{equation*}
    \partial_t\nrm{S_t} \leq \frac{1}2\left|\left|F^*_{tW}\left(g^{j\ell}(\nabla_r\pounds_{W}g_{u\ell}+\nabla_u\pounds_Wg_{r\ell}-\nabla_{\ell}\pounds_Wg_{ur})(\partial_j\otimes dx^u \otimes dx^r)\right)\right|\right|.
\end{equation*}
Then applying the Lie derivative identity $\pounds_Wg_{u\ell} = \nabla_u(W^{\flat}_{\ell})+\nabla_{\ell}(W^{\flat}_u)$ we infer that 
\begin{equation*}
    \nabla_r\pounds_{W}g_{u\ell}+\nabla_u\pounds_Wg_{r\ell}-\nabla_{\ell}\pounds_Wg_{ur} = 2\nabla^2W^{\flat} + 2\text{Rm}\cdot W^{\flat}.
\end{equation*}
The flow can then be extracted from the norm at the expense of introducing another exponential term with rate $\tilde{c}_0$ depending on $W$ in a proportional manner (arising from measuring the flowed metric tensor for the tangent bundle with respect to the unflowed metric as in equation \ref{nrmflow}), thus we find 
\begin{equation*}
    \partial_t\nrm{S_t} \lesssim e^{\tilde{c}_0 t}\left(\nrm{\nabla^2W^{\flat}} + \nrm{\text{Rm}}~\nrm{W}\right) \lesssim e^{\tilde{c}_0 t}\nrm{W}_{\mathcal{C}^2(X,\sigma)}.
\end{equation*}
Since $\tilde{c}_0$ depends on $W$, the ratio $\nrm{W}_{\mathcal{C}^2(X,\sigma)}/\tilde{c}_0$ is independent of any scaling of the vector field $W$, so shrinking $W$ will only cause a decrease in the rate and not an increase in the coefficient of the exponential term. By integrating the ODI, we arrive at
$\nrm{S_t} \lesssim e^{\tilde{c}_0 t}$ for the case that $S_t$ operates on vector fields. Yet, this was assuming $S_t$ acted on vector fields. In general, we may get a multiple $m\tilde{c}_0$ in the exponent to account for the valence of $\beta$, but this poses no threat to the argument. Note the above procedure unfolds similarly for $\nabla^kS_t$, which then depends on $\nrm{W}_{\mathcal{C}^{k+2}(X,\sigma)}.$ 
Taking the largest of the exponents present gives the appropriate $c_1$ for the case at hand.

Now assume we have subexponential bounds with rate $c_k$ for $\nrm{\nabla^kF^*_{tW}\beta}$, so we try to bound the norm of the next derivative:
\begin{equation*}
    \nrm{\nabla^{k+1}F^*_{tW}\beta} \leq \nrm{\nabla^k\nabla_tF^*_{tW}\beta} + \nrm{\nabla^kS_tF^*_{tW}\beta}.
\end{equation*}
Since we have $\nabla_tF^*_{tW}\beta = F^*_{tW}(\nabla\beta)$, we can note $\nabla\beta$ satisfies the hypothesis of having $k$-th derivative subexponential bounds for some rate $\breve{c}_{k}$. Expanding out the other term with the Leibniz Rule gives us
\begin{equation*}
    \nrm{\nabla^{k+1}F^*_{tW}\beta} \lesssim e^{\breve{c}_kt}\nrm{\nabla^{k+1}\beta} + \sum_{j=0}^k \binom{k}{j}\nrm{\nabla^{k-j}S_t}\nrm{\nabla^jF^*_{tW}\beta}.
\end{equation*}
All of these terms have a subexponential growth rate as we have seen; merely take the largest among the resulting addends to define $c_{k+1}$ to complete the induction.
\end{proof}

\begin{proof}[Proof of Proposition \ref{prop:P23}]
We assume we are working about a level set of $\tau$. For the zeroth order estimate, we note that
\begin{equation*}
    |\Psi_V| \leq \left|\int_0^{\tau}F_t^*f_V ~dt \right| \leq \tau\nrm{f_V}_{L^\infty},
\end{equation*}
which is certainly subexponential in $\tau$. For the other values of $k$, we proceed as in the lemma, showing the first derivative explicitly for clarity. First, we define a map from an bounded open subset of $\mathbb{H}\times X$ to $\N$ as follows.
$$\mathcal{A} = \{(\zeta,x)| -1 < \mathfrak{Re}\zeta < 1, 1 < \mathfrak{Im}\zeta < 3\},$$
$$\upsilon_T:\mathcal{A} \rightarrow \N \text{ via } (\zeta,x) \mapsto (e^{ie^T\zeta},x).$$

Note that for a given $T$, the above map is certainly not surjective, but for every $q \in \N$, there is always a value of $T$ for which $q$ is in the image of $\upsilon_T$. Let us now consider $q$ in the level set of interest, setting $T = \lfloor \tau \rfloor$ and taking $\hat{q} \in \upsilon^{-1}_T(q)$. The utility of this map lies in the fact that $\upsilon_T^*\omega = \upsilon^*\omega$ for the covering map defined earlier in equation \ref{upsi}. In other words, this map still pulls $\omega$ back to the standard Poincaré metric on $\mathcal{A}$, so we have the equality
\begin{equation*}
    \nrm{\Psi_V}_{\mathcal{C}^k(\omega)}(q) = \nrm{\upsilon^*_T\Psi_V}_{\mathcal{C}^k(\upsilon^*\omega)}(\hat{q}).
\end{equation*}
However, from one of the fundamental properties of the gnarl shown in \ref{gnarl1}, 
$$\upsilon^*_T\Psi_V = \psi_{(T+\ln\mathfrak{Im}\zeta)V} = \psi_{TV} + F^*_{TV}(\psi_{(\ln\mathfrak{Im}\zeta)V}).$$
Here, the flow is just the trivial extension from $X$ to $\mathcal{A}$. From this relation, we see that for the first derivative, which is just the exterior derivative, we get
\begin{equation*}
    \nrm{d\Psi_V}_\omega(q) = \nrm{d\upsilon^*_T\Psi_V}_{\upsilon^*\omega}(\hat{q}) \leq \nrm{d\psi_{TV}}_{\upsilon^*\omega}(\hat{q}) + \nrm{dF^*_{TV}(\psi_{(\ln\mathfrak{Im}\zeta)V})}_{\upsilon^*\omega}(\hat{q}).
\end{equation*}
For the first term, notice
$$\partial_T\nrm{d\psi_{TV}} \leq \nrm{d(\partial_T\psi_{TV})} = \nrm{dF^*_{TV}f_V} \lesssim e^{c_1T}\nrm{f_V}.$$

The rate $c_1$ arises as in Lemma \ref{flowlem}. Now we can integrate the ODI and find $\nrm{d\psi_{TV}}$ is subexponential. For the term with the flow, observe that we can pull the flow out of the norm at the expense of an exponential factor as in Lemma \ref{flowlem} which is then multiplied by the quantity
$$\nrm{d\psi_{(\ln\mathfrak{Im}\zeta)V})} \leq \sup_{\mathcal{A}} \nrm{d\psi_{(\ln\mathfrak{Im}\zeta)V})},$$
which is independent of $T$, hence why we pulled back to a set whose closure is compact. Together, these give us subexponential bounds for $k=1$.

For higher derivatives, setting $\tilde{\nabla}$ as the connection for $\upsilon^*\omega$, we still have an inequality of the form
\begin{equation*}
    \nrm{\nabla^k\Psi_V}(q) \leq \nrm{\tilde{\nabla}^k\psi_{TV}}(\hat{q})+\nrm{\tilde{\nabla}^kF^*_{TV}(\psi_{(\ln\mathfrak{Im}\zeta)V})}(\hat{q}).
\end{equation*}
The only additional tool needed here is the endomorphism-valued 1-form $\tilde{S}_T = \tilde{\nabla} - F^*_{TV}(\tilde{\nabla})$ which allows commutation of $F^*_{TV}$ and $\tilde{\nabla}$. What is left in the end is a polynomial expression in the $\tilde{\nabla}$ derivatives of $\psi_{TV}, \psi_{(\ln\mathfrak{Im}\zeta)V},$ and $\tilde{S}_T$, the last of which was shown to be subexponential in $T \approx \tau$ in Lemma \ref{flowlem}. Taking the largest exponent among the resulting monomials gives the appropriate $\varepsilon_k$ to complete the proof of the statement.

\end{proof}

Lastly in this section, we make the remark that the metric $\omega_V$ is complete on $\N$. Again appealing to the covering space, we show that $\upsilon^*\omega_V$ is complete by illustrating that the distance to the boundary from any point is infinite. Starting with a point $p = (\zeta_0,x) \in \mathbb{H}\times X$, we can WLOG consider $\zeta_0$ to be purely imaginary of sub-unit length and assume to the contrary of completeness that there exists a geodesic $\gamma:[0,a] \rightarrow \overline{\mathbb{H}\times X}$ which begins at $p$ and terminates at the boundary. Then for $T = -\ln \frac{\zeta_0}i$, which is a positive real number, we have the inequalities
$$\text{length}(\gamma) \geq \text{dist}(p,b(\mathbb{H}\times X)) \geq \sum_{j <  0}\text{dist}(\{e^{jT}i\}\times\R \times X,\{e^{(j-1)T}i\}\times \R \times X),$$
since indeed the geodesic must pass through each slice parameterized by the values of $\mathfrak{Im}\zeta$ on the way to the boundary.

The quantity $\text{dist}(\{e^{jT}i\}\times\R \times X,\{e^{(j-1)T}i\}\times \R \times X)$ is essentially the distance between the boundary components of 
$$\{(\zeta,x) | e^{(j-1)T}< \mathfrak{Im}\zeta < e^{jT}\}.$$
But as we have shown above, the map $\tilde{F}_{(j+1)TV}$ is an isometry for $\omega_V$ that maps this set to $\{(\zeta,x) | e^{-2T}< \mathfrak{Im}\zeta < e^{-T}\}$, which incidentally contains $p$. This implies the distance between the slices is independent of $T$, and since it is positive, each term in the earlier inequalities is infinite, contradicting the assertion that $\gamma$ was of finite length. Hence $\upsilon^*\omega_V$ and likewise $\omega_V$ are complete.

 \section{Proof of Main Result}\label{sec:mainproof}
In order to make clear the hereditary property of extK and cscK PT metrics proven by Auvray in \cite{Auv}, we want to explicitly state the corresponding definition for a class of PT metrics in the global scenario wherein $X$ is a hypersurface with trivial normal bundle in a compact Kähler manifold $(M,\varpi_0)$. The definitions regarding PT metrics are easily extended to $M\setminus X$ by using a bump function $\chi$ around $X$ with the function $\chi\tau$ replacing $\tau$ in all of the statements so that $\varpi_0 +dd^c(\chi\tau)$ becomes the baseline PT metric against which all others are measured.

\begin{definition}\label{def:PT}
  For a smooth Kähler metric $\varpi_0$ on $M$, a PT metric $\varpi$ is said to be in the class $[\varpi_0]$ if there exists a function $y = O(\chi\tau)$ on $M\setminus X$ such that $\varpi = \varpi_0 + dd^cy$ and $\nrm{y}_{\mathcal{C}^k(M\setminus X,\varpi_0 + dd^c{(\chi\tau)})}$ is bounded for all $k > 0$.
\end{definition}

A natural desire for such a family of metrics would be that if the class $[\varpi_0]$ contains an extremal PT metric, then for any closed $(1,1)$-form $\eta$ small enough, $[\varpi_0 + \eta]$ would also contain an extremal PT metric just as in the case for compact Kähler manifolds. However, an extK but not cscK PT metric would need to have an associated nontrivial holomorphic vector field whose extension to $M$ fixes $X$, which does not always exist, even if $X$ admits holomorphic vector fields of its own. In such a case, one would have only a cscK PT metric, yet the perturbed class will restrict to the Kähler class $[(\varpi_0 + \eta)|_{X}]$ on $X$ which could admit extK metrics but no cscK metrics. The following theorem of Auvray, adapted for our present context, implies that in the above situation, no extK PT metrics can exist in $[\varpi_0 + \eta]$.

\begin{thm}\label{thm:heredity}(\cite{Auv}, Theorem 4)
  Let $\varpi_0$ be a Kähler metric on $M$ and assume there exists an extK (resp. cscK) PT metric of class $[\varpi_0]$ on $M\setminus X$. Then there exists an extK (resp. cscK) metric in the class $[\varpi_0|_{X}]$ on $X$.
\end{thm}

This result is proven by showing that locally any extremal PT metric splits near the divisor as $$a~dd^c\tau + \nu^*\sigma + O(|\ln(|z|)|^{-r})$$ for some $a, r > 0$, where $\sigma$ is the resulting extK or cscK metric on $X$. However, the method of proving this is in spirit the reverse of introducing the gnarl. Indeed, Auvray uses the $\tau$-dependent Kähler distortion potential of a cscK PT metric near the divisor to give rise to a family of nearly cscK metrics on $X$ moreover showing that this family is generated by flows of time-$\tau$-dependent holomorphic vector fields.  In any event, the obstruction to our desired LeBrun-Simanca openness is local, and thus we introduce the gnarl as a local construction in order to overcome this issue.

Our approach is to construct an appropriate scalar curvature operator to and from functions on $X$ in accord with the remark at the end of Section 2. Then through linearizing and applying a Banach space implicit function theorem argument, we will arrive at a choice of gnarled PT metric in any nearby Kähler class which is cscK near the divisor and hence has different asymptotic properties than those delineated by Auvray for extK PT metrics. To this end, we need the following lemma.

\begin{lemma}\label{lemma:L1}
The linearization of the map $V \mapsto \text{Scal}(\omega_V)$ at $\omega_V$ is given by $$W \mapsto 2\mathcal{D}_V^{\star_V}\mathcal{D}_V(\tau \underline{f_W}_{\tau}) + 2\langle d\text{Scal}(\omega_V),d(\tau \underline{f_W}_{\tau})\rangle_V,$$ where $\mathcal{D}_V$ is the operator $f \mapsto \bar{\partial}(\gr_V^{1,0}f)$, $\star_V$ and $\langle,\rangle_V$ represent the duality and inner product provided by $\omega_V$, and $\underline{f_W}_{\tau}=\frac{1}{\tau}\int_{0}^{\tau}F_{\varsigma V}^*(f_W+W\psi_{(\tau-\varsigma)V})d\varsigma.$
\end{lemma}

\begin{proof}
We begin with the fact, as shown in \cite{LBS}, that the Fréchet derivative --- defined by convergence on compact subsets --- of the map $\varphi \mapsto $ Scal$(\omega + dd^c\varphi)$ from distortion potentials to scalar curvature is given by $$\phi \mapsto 2\mathcal{D}^\star\mathcal{D}\phi + 2\langle d\text{Scal}(\omega),d\phi\rangle,$$ namely twice the Lichnerowicz operator for $\omega$ plus a correction involving the derivative of the scalar curvature which of course vanishes when based at a cscK metric. The new information here then arises from the linearization of $W \mapsto \Psi_W$ at a gnarled metric, which since the map $\tau\times \nu$ is independent of $W$, we can infer by differentiating the map $W \mapsto d_Xd_X^c\psi_{tW}$. 

As in \cite{Posc}, the time-$t$ flow of a sum of vector fields $V$ and $W$ can be ``factored" into the composition of the time-$t$ flow of $V$ given by $F_{tV}$ and the time-$t$ flow of the time-dependent vector field $F_{tV}^*W$, where $\left(F_{tV}^*W\right)(x) = (F_{tV}^{-1})_*\left(W(F_{tV}(x)\right)$ arises from the adjoint representation of diffeomorphisms acting on vector fields. Because of the double duty of $t$ here, let us briefly call this family of diffeomorphisms $F_{t,F^*_{tV}W}$. Then indeed $F_{tV}\circ F_{t,F^*_{tV}W}$ satisfies the defining equation of $F_{t(V+W)}$ since
\begin{align*}
    \frac{d}{dt}F_{tV}\circ F_{t,F^*_{tV}W}(x) &= V(F_{tV}\circ F_{t,F^*_{tV}W}(x))+F_{tV*}\left((F_{tV}^{-1})_*W(F_{tV}\circ F_{t,F_{tV}^*W}(x)\right) \\
&= V(F_{tV}\circ F_{t,F^*_{tV}W}(x))+W(F_{tV}\circ F_{t,F^*_{tV}W}(x)).
\end{align*}
Then we see for our derivative at $V$ in the direction $W$ that
 \begin{align*}
     \left.\frac{d}{du}\right|_{u=0} d_Xd^c_X \psi_{t(V + uW)} & = \left.\frac{d}{du}\right|_{u=0} F^*_{t(V + uW)}\sigma - \sigma
      =\left.\frac{d}{du}\right|_{u=0} F^*_{t, uF_{tV}^*W}F^*_{tV}\sigma.
 \end{align*}
The map $(t,u,p) \mapsto F_{t,uF^*_{tV}W}(p)$ satisfies 
\begin{equation*}
    \left.\frac{\partial^2}{\partial u\partial t}\right|_{u=0}F_{t,uF^*_{tV}W}(p) = \left.\frac{\partial}{\partial u}\right|_{u=0}uF^*_{tV}W(F_{t,uF^*_{tV}W}(p)) = F^*_{tV}W(p).
\end{equation*}
Thus in our expression pulling back $F^*_{tV}\sigma$, we can integrate along $t$ to find
\begin{equation*}
    \left.\frac{d}{du}\right|_{u=0} F^*_{t, uF_{tV}^*W}F^*_{tV}\sigma = \pounds_{\int_0^tF^*_{\varsigma V}Wd\varsigma}F^*_{tV}\sigma = d\left(\int_0^t F^*_{\varsigma V}W \iprod F^*_{tV}\sigma d\varsigma\right).
\end{equation*}

 Since a vector field is an element of $\h$ if it both vanishes somewhere and is holomorphic, it is apparent that $F^*_{\varsigma V}W \in \h$ for all $\varsigma$. Hence, we can explicitly determine its doubly-time-dependent holomorphy potential for the Kähler metric $F^*_{tV}\sigma$. Recall that since $W\iprod \sigma = d^cf_W$, for any other Kähler metric $\check{\sigma} = \sigma+dd^c\phi$ in the same class as $\sigma$, we have that $W\iprod \check{\sigma}=d^c\left(f_W+W\phi\right)$. Naturally, the holomorphy potential of $F^*_{\varsigma V}W$ with respect to $F^*_{\varsigma V}\sigma$ is $F^*_{\varsigma V}f_W$, so for the metric $F_{tV}^*\sigma$ we add to $F^*_{\varsigma V}f_W$ the derivative in the direction $F^*_{\varsigma V}W$ of the distortion potential $\psi_{tV}-\psi_{\varsigma V}$. This results in 
 $$F^*_{\varsigma V}W \iprod F^*_{tV}\sigma = d^cF^*_{\varsigma V}\left(f_W +W(F^*_{-\varsigma V}(\psi_{tV}-\psi_{\varsigma V}))\right).$$
 Since the operator $d^c$ can be moved past the integral, this derivative is clearly $dd^c$-exact. What is more, as the relationship in \ref{removeH} indicates that $F^*_{-\varsigma V}(\psi_{tV}-\psi_{\varsigma V}) = \psi_{(t-\varsigma)V}$, the integral in the end can be viewed as the average value of the function $F^*_{\varsigma V}\left(f_W+W\psi_{(t-\varsigma)V}\right)$ as $\varsigma$ varies from $0$ to $t$ multiplied afterwards by the weight $t$, which we call $\underline{f_W}_t$.

The complex Hessian operator $d_Xd^c_X$ can be removed just as in equation \ref{removeH}, and the time parameter $t$ can be pulled back to $\tau$, showing in particular that at the origin in $\h$ the derivative of the gnarl is $V \mapsto \tau f_V$. Finally, this can be combined with the formula for the derivative with respect to general distortion potentials to arrive at the desired result, proving the lemma.
\end{proof}

\begin{proof}[Proof of Theorem \ref{thm:T1}]
We will only perturb the metric $\omega$ by perturbing $\sigma$ and pulling back, so in order to parameterize the directions in  $H^{1,1}(X)$ we shall use the real $\sigma$-harmonic $(1,1)$-forms notated $\mathcal{H}_{\sigma}^{1,1}(X,\R)$. Moreover, because taking the scalar curvature will require four derivatives of the distortion potential, we will make of use the Hölder space $\mathcal{C}^{4,\alpha}(X)$ so that scalar curvatures land in $\mathcal{C}^{0,\alpha}(X)$ after restriction to the specific slice $\{e^{-2}\}\times X \cong X$.

Additionally, since we can define a gnarl for any holomorphic vector field in $\mathfrak{h}$ and Kähler metric on $X$, given small enough $\eta$ and $\varphi$, we will now set $\Psi_V^{\eta,\varphi}$ as the gnarl associated to $V$ by $\sigma + \eta + dd^c\varphi$. This permits the definition of a new gnarled metric:
$$\omega_{V,\eta,\varphi} = \omega + \nu^*\eta + dd^c(\nu^*\varphi + \Psi_V^{\eta,\varphi}).$$

The gnarling is so named because it simulates the effect of flowing in directions tangent to $X$ more and more as one approaches the zero section, wrapping around like a gnarl in a tree trunk, while still preserving the holomorphic structure. The gnarl is required here instead of a literal flow since although one can extend $V$ trivially to $N$, the vector field $\tau V$ on $\N$ along which one would flow is not holomorphic, so flowing the original symplectic form would yield something no longer Kähler for the original complex structure. However, the flow of $\tau V$ is still perfectly well-defined here and will allow us to somewhat ``ungnarl" the gnarl by applying the inverse flow to the gnarled metric.

Define $\F$ to be the time-one flow of the vector field $-\tau V$ on $\N$. Then for $J$ the almost complex structure on $\N$, we notice that 
$$G(V,\eta,\varphi) = -\F^*J\iprod \F^*(\omega_{V,\eta,\varphi})$$
defines a Riemannian metric on the smooth manifold $\N$ through the usual procedure of twisting the symplectic form by the almost complex structure to get the symmetric tensor. Such a metric certainly has a scalar curvature, and this is how we shall define our operator:
\begin{align*}
    \mathcal{S}:\h\oplus\mathcal{H}_{\sigma}^{1,1}(X,\R)\oplus\mathcal{C}^{4,\alpha}(X) \rightarrow \mathcal{C}^{0,\alpha}(X)/\R\\
    \text{ via }(V,\eta,\varphi) \mapsto \left[\text{Scal}(G(V,\eta,\varphi))|_{\{e^{-2}\}\times X}\right]
\end{align*}

The goal is to prove that there are $V_{\eta}$ and $\varphi_{\eta}$ such that for any small $\eta$, the metric $G(V_{\eta},\eta,\varphi_{\eta})$ has constant scalar curvature. This metric will indeed be a Kähler metric, but for the new complex manifold $(\N,\F^*J)$ and not our initial $(\N,J)$. The point is that once a constant scalar curvature metric, the Kähler manifold $(\N,\F^{-1*}G(V_{\eta},\varphi_{\eta},\eta),J,\omega_{V_{\eta},\eta,\varphi_{\eta}})$ will also be cscK since the constancy of the scalar curvature will not be altered under a diffeomorphism. Furthermore, even for arbitrarily small $V$, we have shown adding the complex Hessian of the gnarl is not a small perturbation of the original cscK metric, but in fact $G(V,\eta,\varphi)$ is arbitrarily close to the starting metric, so an invertibility argument has some hope of applying.

First we must establish smooth dependence on the variables so that $\mathcal{S}$ is a well-defined continuously differentiable operator near the origin, and then we will invert $\mathcal{S}$ using its linearization and the implicit function theorem. Moving outside in, restriction to a smooth submanifold is a smooth map, and taking the scalar curvature of a Riemannian metric is analytic in the metric's components. Since Hölder continuous functions composed with an analytic one stay Hölder, we now must only establish smoothness for the tensors $\F^*J$ and $\F^*\omega_{V,\eta,\varphi}$, that is to say that their local regularity about $\{e^{-2}\}\times X$ is no worse than that of the inputs.

The pullback of the complex structure is more straightforward to compute, though it has both covariant and contravariant properties. The holomorphic tangent bundle splits into
$$T_{\N} = \mathcal{O}_{\N}\oplus\nu^*T_X$$
where $ \mathcal{O}_{\N}$ is just the trivial line bundle coming from the tangent bundle of $\mathbb{D}^*$ pulled back up the product. Also note that since $\F(z,x) = (z,F_{-\tau V}(x))$, we can split $\F^*J$ up nicely along with $T_{\N}$. That is, for any $(1,0)$ vector field $Y$ on $\N$, 
$$\mathcal{F}_* Y= 			\begin{bmatrix}
				1&0\cdots0\\
				\partial_z F_{-\tau V}&\partial_X F_{-\tau V} 
			\end{bmatrix}Y, \text{ and }
\mathcal{F}_* \bar{Y}= 			\begin{bmatrix}
				1&0\cdots0\\
				\partial_{\bar{z}} F_{-\tau V}&\bar{\partial}_X F_{-\tau V} 
			\end{bmatrix}\bar{Y}.
$$
We note here that by the chain rule, the vector field $\partial_z = \frac{1}{z\ln(|z|)}\partial_{\tau}$, so that $$\partial_z F_{-\tau V} = -\frac{1}{z\ln(|z|)}V.$$ 
We can now use the $J|_X$ equivariance of $d_X F_{-\tau V}$ and the fact that
$$\mathcal{F}^*J = \mathcal{F}_*\circ J \circ \mathcal{F}^{-1}_*$$
to see that for any vector field $Z$ lifted from $X$,
\begin{align*}
   (\mathcal{F}^*J)Z = d_X F_{-\tau V}\circ J\circ d_X F_{\tau V}(Z) = d_X F_{-\tau V}\circ d_X F_{\tau V}(JZ)  = JZ. 
\end{align*}

For vertical vectors, we calculate on the frame for $\mathcal{O}_{\N}\oplus\overline{\mathcal{O}_{\N}}$ given by $\{\partial_z, \partial_{\bar{z}}\}$ to see
\begin{align*}
    (\mathcal{F}^*J)\partial_z &=
    \F_*\left(J\left(\partial_z+\frac{1}{z\ln(|z|)} V\right)\right)
    =\F_*\left(i\partial_z+\frac{1}{z\ln(|z|)}J V\right)
    \\
    &= i\partial_z-\frac{i}{z\ln(|z|)} V +\frac{1}{z\ln(|z|)}J V = i\partial_z -\frac{2i}{z\ln(|z|)} V^{0,1}.
\end{align*}

Here we write $V^{0,1}$ for the $(0,1)$ part of $V$, defined so that $V = V^{1,0}+V^{0,1}$. A similar calculation for $\partial_{\bar{z}}$ shows that as a tensor, we get the following smooth linear expression in terms of $V$:
\begin{equation*}
(\mathcal{F}^*J)=J+4\mathfrak{Re}\left(\frac{i}{\bar{z}\ln(|z|)} V^{1,0}\otimes d\bar{z}\right) = J+4\mathfrak{Re}(iV^{1,0}\otimes \bar{\partial} \tau)
\end{equation*}
$$=J-e^{-\tau}V\otimes d\theta+JV\otimes d\tau = J-V\otimes d^c\tau+JV\otimes d\tau.$$

For the symplectic form, we take note that around the slice we are interested in, $$dd^c\Psi_{V}^{\eta,\varphi} = dd^c(\Psi_{V}^{\eta}+F^*_{\tau V}\nu^*\varphi-\nu^*\varphi).$$
Recalling that $\tau$ is unchanged by $\F$, pulling back the form altogether leaves us with
\begin{equation*}
   \F^*\varpi_{V,\eta,\varphi} = dd^c\tau + \F^*\sigma +\F^*\eta + \mathcal{F}^*dd^c\Psi_{V}^{\eta} + \mathcal{F}^*\left(dd^cF^*_{\tau V}\nu^*\varphi\right). 
\end{equation*}

The map $(V,\eta) \mapsto \mathcal{F}^*(\sigma +\eta)$ is smooth --- clearly in $\eta$ --- with derivative in $V$ behaving as in Lemma \ref{lemma:L1} where the gnarl $\Psi_{V}^\eta$ was shown to be smooth in $V$. Notice that if we use an appropriately normalized Green's operator $\mathfrak{G}_{\sigma}$ for the Laplacian $\Delta_{\sigma}$ of $\sigma$, we can write the gnarl as an affine linear expression in $\eta$:
$$\Psi_{V}^{\eta} = \nu^*\left(\mathfrak{G}_{\sigma}\Lambda_{\sigma}(F^*_{\tau V}(\sigma+\eta)-\sigma-\eta)\right).$$

Lastly then, we can see that the operation applied to $\varphi$ can be simplified to $\varphi \mapsto d(\mathcal{F}^*J)d\nu^*\varphi$, the complex Hessian in terms of the pulled back complex structure. Yet we already showed this flowed complex structure is smooth in $V$ with smooth components, so composing with its linear endomorphism action on differential forms will be smooth in $V$ and $\varphi$. Hence, $\mathcal{F}^*\varpi_{V,\eta,\varphi}$ is a $\mathcal{C}^1$ expression in the variables, and the Riemannian metric $-\mathcal{F}^*J\iprod \mathcal{F}^*(\varpi_{V,\eta,\varphi})$ is likewise so regular, and we can finally conclude the operator $\mathcal{S}$ is also at least $\mathcal{C}^1$.

To finish the proof, we must show that the differential of $\mathcal{S}$ in the variables $V$ and $\varphi$ at $\eta = 0$ is invertible. This will be sufficient for the the local submersion theorem for Banach spaces as stated in Appendix \emph{A3} of Donaldson-Kronheimer \cite{DK} to define a cscK gnarled metric for each nonzero $\eta$ in a small neighborhood $U$ of the origin in $\mathcal{H}^{1,1}(X)$, as specified in the theorem statement. Since the complex Hessian operator is linear in itself, we can replace $\Psi_V$ by its $V$-linearization when differentiating $\mathcal{S}$, as the final pullback by $\F$ introduces only quadratic or higher-order expressions in $V$ which will vanish in the derivative. 

As in Lemma \ref{lemma:L1}, we set $\mathcal{D}^\star\mathcal{D}$ as the Lichnerowicz operator for $\omega$, and additionally we consider $\mathcal{D}_{\sigma}$ and $\star_\sigma$ as the analogous operators for $\sigma$ on $X$. Differentiating in these two variables then gives
\begin{equation*}
    \left.\frac{d}{du}\right|_{u=0}\mathcal{S}(uV,0,u\varphi) = \left[2\mathcal{D}^\star\mathcal{D}(\nu^*(\varphi+\tau f_V))|_{\{e^{-2}\}\times X}\right].
\end{equation*}

Here we refer to Equation 3.4 in the proof of Lemma 4.7 in Sektnan \cite{Sek},  where the problem of blowing up manifolds with PT metrics is considered, particularly with respect to how to handle holomorphic vector fields on $X$. When applied to a function $h$ which is independent of the angle $\theta$, the Lichnerowicz operator on the product of the disk and $X$ expands to
$$\mathcal{D}^\star\mathcal{D}h = (\partial_{\tau}^2-\partial_{\tau})^2h - (\partial_{\tau}^2-\partial_{\tau})\Delta_\sigma h - (\partial_{\tau}^2-\partial_{\tau})h + \mathcal{D}_{\sigma}^\star\mathcal{D}_{\sigma}h.$$
Recall also that $f_V \in \text{ker}\mathcal{D}_{\sigma}^\star\mathcal{D}_{\sigma}$ since it is a holomorphy potential for $V$. Certainly both $\varphi$ and $\tau f_V$ are independent of $\theta$, so 

\begin{equation*}
    \left.\frac{d}{du}\right|_{u=0}\mathcal{S}(uV,0,u\varphi) = \left[2\mathcal{D}^{\star}_{\sigma}\mathcal{D}_{\sigma}\varphi + 2f_V + 2\Delta_{\sigma}f_V\right].
\end{equation*}

Additionally, we can show that this map surjects to $\mathcal{C}^{0,\alpha}(X)/\R$ through an integration by parts due to the self-adjointness of the Lichnerowicz operator on $X$. We note that for any holomorphy potential $f$ on $X$, the expression $f+\Delta_{\sigma}f$ is orthogonal to im$\mathcal{D}_{\sigma}^\star\mathcal{D}_{\sigma}$, forming a complementary basis under the $L^2$-product. Suppose in fact that $f+\Delta_{\sigma}f = \mathcal{D}^{\star}_{\sigma}\mathcal{D}_{\sigma}\xi$ for some $\xi$, then
$$\int_X f^2+\nrm{df}^2\sigma^n = \int_X (f+\Delta_{\sigma}f)f\sigma^n = \int_X f\mathcal{D}^{\star}_{\sigma}\mathcal{D}_{\sigma}\xi\sigma^n = \int_X \xi\mathcal{D}^{\star}_{\sigma}\mathcal{D}_{\sigma}f\sigma^n = 0$$

Thus every function in $\mathcal{C}^{0,\alpha}(X)$ can be written as $\mathcal{D}^{\star}_{\sigma}\mathcal{D}_{\sigma}\varphi + f_V + \Delta_{\sigma}f_V + \text{ const.}$ for some $\varphi$ and $V$, just as we need. This surjectivity allows invocation of the implicit function theorem so that there exists a $U \subset \mathcal{H}_{\sigma}^{1,1}(X,\R)$ containing $0$ such that given $\eta \in U$, we can find $V_{\eta} \in \h$ and $\varphi_{\eta} \in \mathcal{C}^{4,\alpha}(X)$ defining a metric $\varpi_{V_\eta} = \omega + \eta + dd^c\Psi_{V_{\eta}}^{\eta,\varphi_{\eta}}$ which is both cscK and complete on $\N$. 

Since we now know $\Delta_{\varpi_{V_\eta}} (\ln \varpi_{V_\eta}^{n+1})$ is a constant and therefore smooth, by elliptic regularity for the two operators $f \mapsto \Delta_{\varpi_{V_\eta}}f$ and $\varphi \mapsto \ln (\varpi_{V_\eta} + dd^c\varphi)^{n+1}$, one can bootstrap $\varphi_\eta$ to be $\mathcal{C}^{\infty}(X)$. Finally, set $z_0 \in \mathbb{D}^*$ and define $s(x) = (z_0,x)$ to be a fiberwise constant section. Then lastly we clearly have the pullback $$s^*(\varpi_{V_\eta}) = \sigma + \eta + d_Xd_X^c(\Psi_{V_{\eta}}^{\eta,\varphi_{\eta}}|_{\{z_0\}\times X}) \in [\sigma + \eta].$$
\end{proof}

With the existence of these metrics on a product model, a natural next step would be to define their extensions on the complement of a complex hypersurface $X$ in a compact manifold $M$ as is done for PT metrics, or in even greater generality on the complement of a normal crossing divisor in $M$. This introduces various technical complications mainly stemming from the lack of an obvious choice for the extension of a vector field on $X$ to a tubular neighborhood. When there is a holomorphic tubular neighborhood, one can consider a holomorphic lift, but more generally a horizontal lift that preserves the log-polar distance may be preferred. 

Nevertheless, in the case that $X$ has a holomorphically trivial normal bundle, the gnarl $\Psi_V$ can still be introduced as-is to existing PT metrics to construct new gnarled metrics on $M\setminus X$  which are asymptotically cscK without necessarily splitting as a product of cscK metrics near the divisor. The gnarl can even be defined when there is curvature in the normal bundle, using the log-polar distance induced by the metric, and when $V$ is small, the result will still be a Kähler form. The fundamental exploit in the proof of Theorem \ref{thm:T1} is the symmetry provided by the cusp metric on the punctured disk. Without this, one would need to set up the scalar curvature operator using Cheng-Yau weighted Hölder spaces, which requires a starkly different approach than herein followed.

\end{document}